\documentclass[11pt,reqno]{amsart}
\usepackage[T1]{fontenc}
\usepackage{lmodern,amsmath,amssymb,mathtools,booktabs,array,enumitem,microtype}
\usepackage[hidelinks]{hyperref}
\newtheorem{theorem}{Theorem}[section]
\newtheorem{proposition}[theorem]{Proposition}
\newtheorem{lemma}[theorem]{Lemma}
\newtheorem{corollary}[theorem]{Corollary}
\theoremstyle{definition}
\newtheorem{example}[theorem]{Example}

\numberwithin{equation}{section}
\newcommand{\norm}[1]{\left\lVert#1\right\rVert}
\newcommand{\ind}{\chi}
\newcommand{\esssup}{\operatorname*{ess\,sup}}
\newcommand{\distfr}{\operatorname{dist}_{\mathrm{fr}}}
\newcommand{\tactive}{\tau_{\mathrm{active}}}
\newcommand{\calS}{\mathcal S}
\newcommand{\calP}{\mathcal P}
\newcommand{\calQ}{\mathcal Q}
\newcommand{\wideI}{\widehat I}
\newcommand{\dd}{\,d}
\newcommand{\e}{\mathrm e}
\newcommand{\eps}{\varepsilon}
\newcommand{\rank}{\operatorname{rank}}

\setlist[enumerate]{label=\textnormal{(\roman*)},leftmargin=2em}
\hypersetup{pdftitle={Input-Endpoint Compactness and Essential Norms for Weighted Bilinear Hardy Operators},pdfauthor={Saikat Kanjilal}}

\title[Input-endpoint compactness of bilinear Hardy operators]{Input-Endpoint Compactness and Essential Norms for Weighted Bilinear Hardy Operators}
\author{Saikat Kanjilal}
\address{Department of Mathematics, JIS University, Kolkata - 700109, West Bengal, India}
\email{saikat.kanjilal.07@gmail.com}
\urladdr{https://orcid.org/0000-0002-4359-8343}
\date{}
\begin{document}
\begin{abstract}
We study compactness of the product of two weighted Hardy primitives into
$L^q(u)$, $0<q<\infty$, on the input boundary
$1\le p_1,p_2\le\infty$, where at least one input exponent is $1$ or
$\infty$. For $L^1\times L^p$, $1<p<\infty$, boundedness implies
compactness when $q<p/(p+1)$. For $L^1\times L^1$, the corresponding
range is $q<1/2$. At and above these boundaries compactness is
characterized by additional active-tail decay, and critical power
weights give bounded noncompact examples. The endpoint reductions use
cumulative Stieltjes measures with their initial atoms and strict
Hardy--Copson brackets. An output-mass allocation argument proves local
finite-rank density, including quasi-Banach targets. Applying the
localization principles developed for the interior-input regime identifies
the essential distance, finite-rank distance and Hausdorff noncompactness
radius with the coupled active tail, with the appropriate metric power
when $q<1$. For a finite left endpoint, ordinary $L^\infty$ inputs admit
an exact scalar transfer completed by linear nodal approximation.
\end{abstract}

\maketitle
\section{Introduction}\label{sec:intro}

The weighted inequality for the product of two Hardy primitives asks
when
\begin{equation}\label{eq:introinequality}
 \left(\int_a^b\left|\int_a^xf(t)\,dt\right|^q
                    \left|\int_a^xg(t)\,dt\right|^qu(x)\,dx\right)^{1/q}
 \le C\norm f_{L^{p_1}(v_1)}\norm g_{L^{p_2}(v_2)}.
\end{equation}
Its boundedness theory has developed through direct weighted criteria,
iteration, and reductions to linear Hardy-type inequalities. The work
of Aguilar Ca\~nestro, Ortega Salvador and Ram\'irez Torreblanca
\cite{AOR2012} and the iteration method of K\v repela
\cite{Krepela2017} are important points of reference. Broad input
endpoint boundedness is already present in the multidimensional theory
of Bilgi\c cli, Mustafayev and \"Unver \cite{BMU2020}. The question
addressed here is which bounded operators in \eqref{eq:introinequality}
have relatively compact images of the product unit ball when at least
one input exponent is $1$ or $\infty$.

The distinction between boundedness and compactness is particularly
visible at an $L^1$ input. For the mixed face $L^1\times L^p$,
$1<p<\infty$, and for the double endpoint $L^1\times L^1$, the
relevant compactness thresholds are respectively
\begin{equation}\label{eq:introthresholds}
 q=\frac p{p+1}\qquad\hbox{and}\qquad q=\frac12.
\end{equation}
Below these values boundedness implies compactness. At and above them
additional decay of the endpoint tails is necessary, and bounded
noncompact examples occur. The numerical boundaries reflect the
harmonic exponent relation already familiar from boundedness theory.
The result concerns the compactness behavior around those boundaries,
including the endpoint realization required to prove it.

Our scope is exactly the finite-$q$ input boundary
\[
 0<q<\infty,\qquad 1\le p_1,p_2\le\infty,\qquad
 \{p_1,p_2\}\cap\{1,\infty\}\ne\varnothing.
\]
The full $L^\infty$ faces are considered only when $a>-\infty$.
At these inputs we use the ordinary essential-supremum norm. This
convention and interval qualification are needed for a meaningful
scalar transfer. A density-weighted $L^p$ norm at finite $p$ must not
be continued to a multiplicatively weighted supremum without changing
the operator problem.

\subsection{Endpoint structure}
For $1<p<\infty$ the primitive-evaluation norm is a dual-weight
integral. At $p=1$ it becomes the running essential supremum
\[
 A_1^-(x)=\esssup_{a<t<x}v(t)^{-1}.
\]
The resulting monotone function can have jumps, a singular continuous
part, and a nonzero initial limit. The quasi-Banach endpoint reduction
therefore naturally involves a cumulative Stieltjes measure. In
particular, a constant nonzero evaluation profile contributes an
initial atom. Deleting this mass makes even elementary nonzero
operators disappear from the transformed criterion.

We construct a compactness-ready representation that keeps these
masses and the ordered endpoint brackets. For $L^1\times L^p$, the
two scalar reductions lead to a positive Hardy--Copson operator from
$L^{1/q}(\mu)$ to $L^{p/(p-q)}(\nu)$ when $q<1$. For two $L^1$
inputs, both measures are of endpoint Stieltjes type and the
transformed exponents are $1/q$ and $1/(1-q)$. The comparison of the
two transformed exponents gives \eqref{eq:introthresholds}.
In the lower regime the integral characteristics disappear under
shrinking output tails. The proof must establish integrability first
and preserve the possible initial output mass in the limiting step.

The compactness conclusion also needs local approximation. The
interior argument based on uniformly small dual-weight increments
cannot simply be substituted at $p=1$. A running essential supremum
need not have small increments uniformly over a fine partition.
Instead we control the error of a primitive sampler by distributing
input masses against the finite atomless measure
$u(x)/U(x)\,dx$, where $U(x)=\int_x^{b_u}u$ and $b_u$ is the active
right endpoint of the target. This proves local finite-rank density
for finite input endpoints, including quasi-Banach targets.
The sampler formula itself is related to the interior construction.
The endpoint content is its error estimate and mass allocation.

The $L^\infty$ section has a different mechanism. Normalizing a
primitive by $x-a$ produces a uniformly bounded multiplier that is
uniformly Lipschitz on each strict middle interval. The remaining
factor is the scalar map $T_pg=(x-a)Hg$. A fixed linear nodal
interpolant for the multiplier, multiplied by a finite-rank linear
scalar sampler, gives a bilinear finite-rank approximant. This
realization upgrades the exact equality of all output-set norms to
exact quantitative compactness transfer. It covers $p=1$ as well as
$p>1$. The double-$L^\infty$ corner is an explicit completion corollary.

\subsection{Relation to existing theory}
The reductions rely on established scalar Hardy criteria, particularly
Sinnamon--Stepanov \cite{SinnamonStepanov1996}, the normal-form
framework of Sinnamon \cite{Sinnamon2022}, and the abstract $p=1$
theory of Santacruz Hidalgo \cite{Santacruz2024}. General-measure
Hardy and Copson theorems are background here. Their application
requires translating strict and inclusive sets and retaining masses,
including a left-initial mass on an adjoined ordered endpoint.
The density superposition results of Gogatishvili, Mihula, Pick,
Hana Tur\v cinov\'a and \"Unver \cite{GMP2022}, together with
Gogatishvili--\"Unver \cite{GU2024,GU2025}, provide closely related
boundedness machinery. We do not interpret a density theorem as an
atomic-measure theorem without the normal-form translation.

The linear compactness literature already gives substantial guidance.
Edmunds--Gurka--Pick \cite{EGP1994} and Edmunds--Stepanov
\cite{ES1994} study compactness and quantitative approximation of
linear Hardy-type or Volterra maps. Eveson \cite{Eveson1995} treats
linear integral-operator compactness at $L^1$ and $L^\infty$.
Lomakina \cite{Lomakina2009} develops linear Hardy approximation in
quasi-Banach targets, and Lomakina--Nasyrova \cite{LN2026} establish
Lorentz-space compactness and approximation results. These works
preclude treating either endpoint compactness or quasi-Banach sampling
as a new general principle. The local upper bounds in this paper are
specific bilinear endpoint estimates. No sharp global approximation
asymptotics are claimed.

Further nearby developments concern distinct operators or conclusions.
Stepanov--Shambilova \cite{SS2019} study iterated and bilinear
Hardy-type boundedness with kernels. Stepanov--Ushakova \cite{SU2022}
treat a linear rectangular two-dimensional Hardy operator, whose
output geometry differs from the same-variable product studied here.
Mohanty--Jain--Jain \cite{MJJ2025} characterize boundedness for a
mixed-direction Hardy product. Recent kernel and discrete boundedness
results include Kalybay--Temirkhanova \cite{KT2026} and
Oinarov--Persson--Omarbayeva--Singh \cite{OPOS2026}.
Edmunds--Meskhi--Natelashvili \cite{EMN2025} give boundedness criteria
for multilinear Riemann--Liouville operators, providing additional
recent multilinear background. None of these boundedness results is
used as a compactness theorem for the present endpoint map.
General multilinear noncompactness also has an independent
interpolation theory. The recent work of Masty\l o and Brandani da
Silva \cite{MS2026} belongs to that setting. Our exact quantities come
from a localized tail representation rather than an interpolation
estimate.

Kanjilal \cite{B1} treats the interior input range
$1<p_1,p_2<\infty$ and develops the abstract localization principles
used here. The present paper concerns the input boundary and supplies
the endpoint-specific reductions, local approximation mechanisms,
compactness phase behavior and examples required to apply that
abstract theory. We state the inherited result once and cite its
proof. Thus the contribution is the endpoint realization, not the
abstract localization theorem or its essential-distance and
noncompactness arguments. To the best of our knowledge, the combined
endpoint compactness theory developed here does not appear in the
existing weighted bilinear Hardy literature.

\subsection{Organization}
Section~\ref{sec:framework} fixes the endpoint conventions and proves
local realization. Sections~\ref{sec:mixed} and~\ref{sec:double}
develop the two $L^1$ theorem blocks and their examples.
Section~\ref{sec:infinity} completes the $L^\infty$ faces.
Section~\ref{sec:synthesis} collects the quantitative conclusions and
the eight-cell phase table, distinguishing quasi-norm radii from
complete-metric radii.

\section{Endpoint framework and local realization}\label{sec:framework}

Let $I=(a,b)$, where $-\infty\le a<b\le\infty$. The input densities
$v_i$ are measurable and satisfy $0<v_i<\infty$ almost everywhere.
The output density $u$ is nonnegative and finite almost everywhere. It may
vanish on sets of positive measure. We use
\[
 X_i=L^{p_i}(v_i),\qquad
 \norm{f}_{L^p(v)}=\left(\int_I|f|^pv\right)^{1/p}\quad(1\le p<\infty).
\]
For $p=\infty$ the norm is the ordinary essential supremum of $|f|$.
In particular, the density is not multiplied into the essential supremum.
Throughout the paper $0<q<\infty$, and at least one of $p_1,p_2$ is $1$
or $\infty$. Define
\begin{equation}\label{eq:operator}
 Hf(x)=\int_a^x f(t)\,dt,\qquad B(f,g)(x)=Hf(x)Hg(x).
\end{equation}
The notation $B$ represents the operator denoted by $\mathcal H_2$ in
\cite{B1}. All output functions are equivalence classes for $u(x)\,dx$.

If $u=0$ almost everywhere, we define the target map to be zero and all
quantities below vanish. Otherwise set
\begin{equation}\label{eq:active}
 \begin{gathered} b_u=\inf\{d\in[a,b]:u=0\text{ a.e. on }(d,b)\},\\
 I_u=(a,b_u),\qquad U(x)=\int_x^{b_u}u(y)\,dy.\end{gathered}
\end{equation}
Only primitives on $I_u$ are needed. Where at least one input is
$L^\infty$, we assume $a>-\infty$. For finite inputs our local
admissibility hypothesis is that the evaluation profiles below are finite
at every $a<x<b_u$. This guarantees absolutely convergent primitives of
all inputs at those cuts.

\begin{proposition}[Evaluation profiles and cumulative measures]\label{prop:profiles}
The norm of $f\mapsto Hf(x)$ on the indicated input space is
\begin{equation}\label{eq:profiles}
 \begin{aligned}
 A_1^-(x)&=\esssup_{a<t<x}v(t)^{-1},\\
 A_p(x)&=\left(\int_a^xv(t)^{1-p'}\,dt\right)^{1/p'}
       &&(1<p<\infty),\\
 A_\infty(x)&=x-a &&(a>-\infty).
 \end{aligned}
\end{equation}
Write $A_{i,p_i}$ for the profile belonging to input $i$, using $A_{i,1}^-$
at an $L^1$ input. Nontrivial boundedness implies local admissibility and
\begin{equation}\label{eq:pointtest}
 0<U(x)<\infty,\qquad
 U(x)^{1/q}A_{1,p_1}(x)A_{2,p_2}(x)\le\norm B
 \quad(a<x<b_u).
\end{equation}
For an $L^1$ profile let
$\bar A_1(x)=\lim_{y\downarrow x}A_1^-(y)$ at ordinary cuts and
$\bar A_1(a)=\lim_{x\downarrow a}A_1^-(x)$. On the ordered interval
$\wideI=[a,b_u)$, with an adjoined left point if necessary, the measure
$d(\bar A_1^r)$, $r>0$, means
\begin{equation}\label{eq:stieltjes}
 \mu(\{a\})=\bar A_1(a)^r,\quad
 \mu((s,t])=\bar A_1(t)^r-\bar A_1(s)^r,
 \quad\mu([a,x])=\bar A_1(x)^r.
\end{equation}
It is locally finite and sigma-finite under local admissibility. For
$1<p<\infty$, $A_p$ is continuous, $A_p(a+)=0$, and $d(A_p^R)$ has no
atoms for every $R>0$.
\end{proposition}

\begin{proof}
Weighted H\"older gives the finite-$p>1$ upper bound and its dual
extremizing sequence gives equality. For $p=1$, write the primitive as
$\int f v\,v^{-1}$. Concentration on positive-measure subsets where
$v^{-1}$ approaches its essential supremum gives equality. For
$p=\infty$, the function $1$ attains the value $x-a$.
Choose nonnegative unit inputs supported before $x$ whose primitives
approach the two evaluation norms. Their primitives remain constant
beyond $x$, and integration of their product over $(x,b_u)$ proves
\eqref{eq:pointtest}. Truncated inputs also prove that an infinite
profile is incompatible with nontrivial boundedness. The positivity of
the other profile and the active-tail property then force finiteness of
$U(x)$.

A monotone profile has at most countably many jumps. Its right-continuous
envelope agrees with it almost everywhere for Lebesgue measure and
defines \eqref{eq:stieltjes}. Finiteness follows by comparison with a
strictly later finite profile. Exhausting $\wideI$ by initial intervals
proves sigma-finiteness. Finally, $v^{1-p'}$ is integrable on every
initial interval with endpoint below $b_u$. Continuity of its integral
proves both continuity and the zero initial limit of $A_p$. The same
statements imply continuity of $A_p^R$.
\end{proof}

The representative convention is part of the measure definition. For
$v=1$, $A_1^-=1$ on $I_u$, and \eqref{eq:stieltjes} gives
$\mu=\delta_a$, not the zero measure. At an interior jump, the physical
left-open profile and its cumulative representative need not have the
same value. Their agreement in an outer density integral does not
permit their interchange inside a Stieltjes integral.
If $a=-\infty$, the singleton $\{-\infty\}$ belongs to the adjoined
ordered space, not to $\mathbb R$. Before invoking a Borel-measure
result on $\mathbb R$, push both measures forward by the increasing map
$x\mapsto e^x$, with $-\infty\mapsto0$. Composition with its inverse is
an isometry of the corresponding $L^s$ spaces and preserves all lower
and upper sets, including their brackets and masses.

\subsection{Quantities inherited from localization}
A bounded bilinear map is compact when the image of the product of the
closed unit balls has compact closure. Its rank is the dimension of its
output span. For such a map $T$ into $Y=L^q(u)$ define
\[
 \begin{gathered}
 \norm T_\e=\inf_{K\ \mathrm{compact}}\norm{T-K},\qquad
 \distfr(T)=\inf_{F\ \mathrm{finite\ rank}}\norm{T-F},\\
 a_n^{(2)}(T)=\inf_{\rank F<n}\norm{T-F}.
 \end{gathered}
\]
All approximating maps in these definitions are bounded and bilinear.
Let $\chi_q(T)$ be the infimum of radii of finite covers of the unit-ball
image by open $Y$ quasi-norm balls with centers in $Y$. For $q<1$ the
complete metric is $\rho(y,z)=\norm{y-z}_q^q$. Its radius quantity is
written $\chi_\rho(T)$.

For $a<c<d<b_u$ use the multiplication operators
\begin{equation}\label{eq:cutoffs}
 \begin{gathered}M_{c,d}h=\ind_{(c,d)}h,\qquad
 R_{c,d}h=\ind_{(a,c)\cup(d,b_u)}h,\\
 \tactive(T)=\inf_{a<c<d<b_u}\norm{R_{c,d}T}.\end{gathered}
\end{equation}
The target identifies $L^q(u;I_u)$ isometrically with $L^q(u;I)$.
Thus the abstract results of \cite{B1} apply on $I_u$ without changing
that work's setting. In its notation, $T^{c,d}$ is our $M_{c,d}T$ and
$Q_{c,d}T$ is our $R_{c,d}T$.

\begin{theorem}[Inherited localization, {\cite[Theorems 5.1 and 5.3]{B1}}]
\label{thm:inherited}
Let $T:E_1\times E_2\to L^q(u;I_u)$ be bounded, with normed input
spaces and $0<q<\infty$. If every $M_{c,d}T$ is compact, then
\[
 T\text{ compact}\ \Longleftrightarrow\ \tactive(T)=0,
 \qquad \norm T_\e=\chi_q(T)=\tactive(T).
\]
Vanishing of the coupled tail is equivalent to vanishing of both
one-sided operator-norm tails. If each $M_{c,d}T$ is a norm limit of
finite-rank bilinear maps, then also
\[
 \distfr(T)=\lim_{n\to\infty}a_n^{(2)}(T)=\tactive(T).
\]
For $q<1$, $\chi_\rho(T)=\tactive(T)^q$.
\end{theorem}

This is an attributed statement of the abstract infrastructure. Its
proof is not repeated. The next lemma supplies its hypothesis at finite
input endpoints independently of the Stieltjes boundedness criteria.

\subsection{Finite input exponents}
\begin{lemma}[Endpoint local finite-rank realization]\label{lem:sampler}
Suppose $1\le p_1,p_2<\infty$, with at least one input exponent $1$,
and $B:X_1\times X_2\to L^q(u)$ is bounded. Each strict-middle map
$M_{c,d}B$ is an operator-norm limit of finite-rank bilinear maps.
For a partition $\mathcal P:c=x_0<\cdots<x_N=d$, put
\begin{equation}\label{eq:sampler}
 S_{\mathcal P}(f,g)=\sum_{j=1}^N
 Hf(x_{j-1})Hg(x_{j-1})\ind_{J_j},\qquad J_j=(x_{j-1},x_j).
\end{equation}
The rank is at most $N$. On the fixed window, set
\begin{equation}\label{eq:gauge}
 d\nu_0(x)=\frac{u(x)}{U(x)}\,dx,\qquad
 V=\nu_0((c,d))=\log\frac{U(c)}{U(d)}<\infty.
\end{equation}
Partitions can be chosen so that, for $0<q<1$,
\begin{equation}\label{eq:localrate}
 a_{N+1}^{(2)}(M_{c,d}B)^q
 \le C_q\norm B^q V\bigl(N^{-q/p_1}+N^{-q/p_2}\bigr).
\end{equation}
\end{lemma}

\begin{proof}
The coefficients in \eqref{eq:sampler} are products of bounded linear
functionals, and its outputs belong to the span of the $N$ indicators.
These indicators are in $L^q(u)$ by $U(c)<\infty$. Formula
\eqref{eq:sampler} is the familiar primitive sampler also used in the
interior construction of \cite{B1}. We prove its endpoint error estimate.

For unit inputs $f_1=f,f_2=g$, let
$m_{ij}=\int_{J_j}|f_i|^{p_i}v_i$, so $\sum_jm_{ij}\le1$.
For $x\in J_j$, the endpoint evaluation inequality gives
\[
 \left|\int_{x_{j-1}}^x f_i\right|
 \le A_{i,p_i}(x)m_{ij}^{1/p_i},\qquad
 |Hf_i(x)|,\ |Hf_i(x_{j-1})|\le A_{i,p_i}(x).
\]
The product difference can be written as
\[
 [Hf(x)-Hf(x_{j-1})]Hg(x)
 +Hf(x_{j-1})[Hg(x)-Hg(x_{j-1})].
\]
Using $(s+t)^q\le C_q(s^q+t^q)$ for nonnegative $s,t$, with
$C_q=\max(1,2^{q-1})$, and then \eqref{eq:pointtest}, we obtain
\begin{equation}\label{eq:pointwiseerror}
 |B(f,g)(x)-S_{\mathcal P}(f,g)(x)|^q
 \le C_q\norm B^q U(x)^{-1}
       \sum_{i=1}^2m_{ij}^{q/p_i}.
\end{equation}
Integration and then maximization over the input masses give the
precise operator-norm interpretation of the allocation estimate
\begin{equation}\label{eq:allocationerror}
 \norm{M_{c,d}B-S_{\mathcal P}}^q
 \le C_q\norm B^q\sum_{i=1}^2
 \sup_{\substack{m_j\ge0\\\sum_jm_j\le1}}
 \sum_j\nu_0(J_j)m_j^{q/p_i}.
\end{equation}
For $0<\gamma<1$ and $a_j\ge0$, H\"older in the finite sum gives
\begin{equation}\label{eq:allocation}
 \sup_{\sum m_j\le1}\sum_ja_jm_j^\gamma
 =\left(\sum_ja_j^{1/(1-\gamma)}\right)^{1-\gamma}
 \le (\max_ja_j)^\gamma\left(\sum_ja_j\right)^{1-\gamma}.
\end{equation}
Equality in the first formula follows by taking $m_j$ proportional to
$a_j^{1/(1-\gamma)}$. For $\gamma\ge1$ the same supremum is at most
$\max_ja_j$. The finite measure $\nu_0$ is atomless, so its cumulative
function allows partitions with maximal cell mass tending to zero.
Equations \eqref{eq:allocationerror}--\eqref{eq:allocation} prove
operator-norm convergence for every $q>0$. No target Minkowski
inequality has been used when $q<1$.

If $V=0$ the middle map is zero. Otherwise choose equal-$\nu_0$-mass
cells, allowing zero-density gaps in the cuts. Then
$\nu_0(J_j)=V/N$. For $q<1$ both $q/p_i<1$, so
\eqref{eq:allocation} equals $VN^{-q/p_i}$. The rank bound gives
\eqref{eq:localrate}.
\end{proof}

In particular, for $(p_1,p_2)=(1,p)$ the right side of
\eqref{eq:localrate} contains $N^{-q}+N^{-q/p}$. For $(1,1)$ its
$q$th root is $O(N^{-1})$ on a fixed window. These are local upper
estimates. The data and the window enter through $\norm B$ and $V$.
There is no global power law asserted here. The argument uses the
atomlessness of $u/U\,dx$, not atomlessness of the Stieltjes input
measures. This distinction permits jumps and singular continuous parts
in the endpoint profiles.

\section{Mixed \texorpdfstring{$L^1\times L^p$}{L1 times Lp} inputs}\label{sec:mixed}

Fix $1<p<\infty$. Write $A_1^-$ for the first input profile and $A_p$
for the second. The setting and local admissibility are those of
Section~\ref{sec:framework}. For an output density $0\le w\le u$ put
$W(x)=\int_x^{b_u}w$ and denote the corresponding map by $B_w$.
The fixed endpoint $b_u$ is retained even when $w$ is a disconnected
output truncation. This convention avoids a change of measures when
taking tails.

For $0<q<1$ set
\begin{equation}\label{eq:mixedparams}
 \begin{gathered}r=\frac q{1-q},\quad R=\frac{pq}{p-q},\quad
 \theta=\frac1q,\quad\theta'=\frac1{1-q},\quad t=\frac p{p-q},
 \\ d\mu=d(\bar A_1^r),\qquad d\nu=d(A_p^R).\end{gathered}
\end{equation}
All cumulative measures follow Proposition~\ref{prop:profiles}. Define
\begin{equation}\label{eq:mixedS}
 \calS_Wh(x)=W(x)\int_{[a,x]}h\dd\mu
                +\int_{(x,b_u)}W(s)h(s)\dd\mu(s).
\end{equation}
We denote the first and second terms by $\calP_Wh$ and $\calQ_Wh$.
The following explicit characteristics specify the result without a
supremum over a remaining original input.

If $\theta\le t$, let
\begin{align}
 \mathfrak P_m(w)&=\sup_{x\in\wideI}\bar A_1(x)^q
       \left(\int_{[x,b_u)}W(y)^t\dd\nu(y)\right)^{1/t},
       \label{eq:mixedPs}\\
 \mathfrak Q_m(w)&=\sup_{x\in\wideI}\nu([a,x))^{1/t}
       \left(\int_{[x,b_u)}W(y)^{\theta'}\dd\mu(y)\right)^{1/\theta'}.
       \label{eq:mixedQs}
\end{align}
If $t<\theta$, put $\eta=(1/t-1/\theta)^{-1}=p/[p(1-q)-q]$ and let
\begin{align}
 \mathfrak P_m(w)^\eta
 &=\int_{\wideI}
 \left(\int_{[x,b_u)}W(y)^t\dd\nu(y)\right)^{\eta/\theta}
 \mu([a,x])^{\eta/\theta'}W(x)^t\dd\nu(x),\label{eq:mixedPi}\\
 \mathfrak Q_m(w)^\eta
 &=\int_{\wideI}\nu([a,x])^{\eta/\theta}
 \left(\int_{(x,b_u)}W(y)^{\theta'}\dd\mu(y)\right)^{\eta/\theta'}
 \dd\nu(x).\label{eq:mixedQi}
\end{align}
Set $\mathcal C_m(w)=(\mathfrak P_m(w)+\mathfrak Q_m(w))^{1/q}$.
For $q\ge1$, instead define
\begin{equation}\label{eq:mixedBanachC}
 \mathcal C_m(w)=\esssup_{s\in I}
 \frac1{v_1(s)}
 \norm{H:L^p(v_2)\longrightarrow L^q(w\ind_{(s,b_u)})}.
\end{equation}
An empty target interval contributes zero. Nonnegative products in these
characteristics use $0\cdot\infty=0$, as justified by monotone
truncation on the contributing sets.

\begin{theorem}[Main Theorem A]\label{thm:mixed}
For $1<p<\infty$ and $0<q<\infty$, $B:L^1(v_1)\times L^p(v_2)
\to L^q(u)$ is bounded if and only if $\mathcal C_m(u)<\infty$.
For $0<q<1$, uniformly for $0\le w\le u$,
\begin{equation}\label{eq:mixednorm}
 \norm{B_w}^q\asymp_{p,q}
 \norm{\calS_W:L^\theta(\mu)\to L^t(\nu)},\qquad
 \norm{B_w}\asymp_{p,q}\mathcal C_m(w).
\end{equation}
For $q\ge1$, $\norm{B_w}=\mathcal C_m(w)$ exactly.
In the bounded case, compactness is equivalent to
\begin{equation}\label{eq:mixedtails}
 \lim_{c\downarrow a}\mathcal C_m(u\ind_{(a,c)})=0,
 \qquad
 \lim_{d\uparrow b_u}\mathcal C_m(u\ind_{(d,b_u)})=0.
\end{equation}
In particular,
\begin{equation}\label{eq:mixedautomatic}
 0<q<\frac p{p+1}\quad\Longrightarrow\quad
 B\text{ bounded}\ \Longleftrightarrow\ B\text{ compact}.
\end{equation}
At and above $p/(p+1)$ compactness is nonautomatic. Interchanging the
inputs gives the same statements for $L^p(v_1)\times L^1(v_2)$.
\end{theorem}

\subsection{Endpoint reduction}
We first record the scalar identity responsible for the initial mass.
The scalar $L^1$ Hardy criterion of Sinnamon--Stepanov
\cite[Theorems 3.2--3.3]{SinnamonStepanov1996}, with the ordered-core
formulation in \cite{Santacruz2024}, says that for a density $k\ge0$,
$F(x)=\int_x^{b_u}k$, and $0<q<1$,
\begin{equation}\label{eq:L1scalar}
 \norm{H:L^1(v)\to L^q(k)}
 \asymp_q\left(\int_{I_u}(A_1^-(y)F(y))^r k(y)\,dy\right)^{1/r}.
\end{equation}
The density statement extends to the present ordered interval by
increasing changes of variable and truncation. It is used with density
outputs here, not as a theorem about arbitrary atomic output measures.
The latter appear only after the following exact measure identity.
For $G=\bar A_1^r$ and $d\mu=dG$ including its initial mass, Tonelli gives
\begin{equation}\label{eq:tonelliidentity}
 \begin{aligned}
 \int_{I_u}G(y)F(y)^rk(y)\,dy
 &=\int_{\wideI}\int_{s}^{b_u}F(y)^rk(y)\,dy\dd\mu(s)\\
 &=\frac1{r+1}\int_{\wideI}F(s)^{r+1}\dd\mu(s).
 \end{aligned}
\end{equation}
For bounded tails the inner equality follows from $-F'=k$ and
$F(b_u-)=0$. Applying it first to bounded compactly truncated densities
and passing monotonically proves the extended-valued identity. Thus no
unproved boundary cancellation is required, and an initial atom is
counted exactly once.

Fix $g$ and take $k=|Hg|^qw$. With
$F_g(x)=\int_x^{b_u}|Hg(y)|^qw(y)\,dy$, equations
\eqref{eq:L1scalar}--\eqref{eq:tonelliidentity} and
$r+1=\theta'$ give
\begin{equation}\label{eq:firstmixed}
 \sup_{\norm f_{L^1(v_1)}\le1}\norm{Hf\,Hg}_{L^q(w)}^q
 \asymp_q\norm{F_g}_{L^{\theta'}(\mu)}.
\end{equation}
Since $\theta'>1$, this intermediate space has ordinary Banach duality.
For $h\ge0$ put $\Phi_h(y)=\int_{[a,y]}h\dd\mu$. Then
\begin{equation}\label{eq:dualstep}
 \norm{F_g}_{L^{\theta'}(\mu)}
 =\sup_{\substack{h\ge0\\\norm h_{L^\theta(\mu)}\le1}}
 \int_{I_u}|Hg(y)|^qw(y)\Phi_h(y)\,dy.
\end{equation}
This is duality in an auxiliary Banach space, not duality in $L^q(w)$.
The integration exchange is Tonelli for positive functions.

For each $h$, the usual scalar Hardy integral criterion for $q<p$
\cite{SinnamonStepanov1996,Sinnamon2022} yields
\begin{equation}\label{eq:secondmixed}
 \begin{gathered}\norm{H:L^p(v_2)\to L^q(w\Phi_h)}^q
 \asymp_{p,q}
 \left(\int_{\wideI}Z_h(x)^t\dd(A_p^R)(x)\right)^{1/t},
 \\ Z_h(x)=\int_x^{b_u}w(y)\Phi_h(y)\,dy.\end{gathered}
\end{equation}
The initial value of $A_p^R$ is zero. This Stieltjes version follows
from the scalar integral criterion by Tonelli. Local admissibility
ensures that the dual-weight primitive is finite at every strict cut.
Taking the two suprema in \eqref{eq:dualstep} in either order and using
\eqref{eq:secondmixed} gives the first equivalence in
\eqref{eq:mixednorm}. Indeed, Tonelli also gives $Z_h=\calS_Wh$ in
\eqref{eq:mixedS}. The atom at the cut belongs to the lower term.

\subsection{General-measure characteristics and their translation}
For clarity we state the background used to evaluate this last norm.
Let $1<P,S<\infty$, let $\sigma,\lambda$ be positive sigma-finite
ordered Borel measures, and put
\[
 \mathsf H_\sigma h(x)=\int_{[a,x]}h\dd\sigma,
 \qquad \mathsf Q_{K,\sigma}h(x)=\int_{(x,b_u)}Kh\dd\sigma.
\]
Under finite relevant one-sided masses, their norms from $L^P(\sigma)$
to $L^S(\lambda)$ are equivalent, respectively, to
\begin{align}
 \sup_{x\in\wideI}\sigma([a,x])^{1/P'}\lambda([x,b_u))^{1/S},
 &\quad P\le S,\label{eq:generalHs}\\
 \left(\int_{\wideI}\lambda([x,b_u))^{\zeta/P}
        \sigma([a,x])^{\zeta/P'}\dd\lambda(x)\right)^{1/\zeta},
 &\quad S<P,\label{eq:generalHi}
\end{align}
and
\begin{align}
 \sup_{x\in\wideI}\lambda([a,x))^{1/S}
 \left(\int_{[x,b_u)}K^{P'}\dd\sigma\right)^{1/P'},
 &\quad P\le S,\label{eq:generalQs}\\
 \left(\int_{\wideI}\lambda([a,x])^{\zeta/P}
 \left(\int_{(x,b_u)}K^{P'}\dd\sigma\right)^{\zeta/P'}\dd\lambda(x)
 \right)^{1/\zeta},&\quad S<P,\label{eq:generalQi}
\end{align}
where $1/\zeta=1/S-1/P$. These are consequences of
\cite[Theorems 2.4, 2.6, 2.7, 3.8 and 3.9]{Sinnamon2022}, with
constants depending only on $P,S$.

Here are the translation details needed at endpoints. After the ordered
pushforward of Section~\ref{sec:framework}, the lower sets are real
Borel sets of the inclusive type in Theorem 2.6 of that source. Its
normal-form reduction preserves the operator norm. An atom of the
output measure becomes an interval of its exact mass in the rearranged
variable. Integrating a nonnegative power over such an interval is
comparable to its inclusive cumulative endpoint value times the mass,
with constants depending only on the power. Applying the supremal and
integral normal-form criteria and summing these contributions yields
\eqref{eq:generalHs}--\eqref{eq:generalHi}. Singular continuous parts
are preserved by the same distribution identity.

For the upper operator absorb its kernel into
$d\tau=K^{P'}\dd\sigma$. With
$h=fK^{1-P'}$ on $\{K>0\}$ and $h=0$ on $\{K=0\}$, the literal identity is
\begin{equation}\label{eq:kernelzero}
 \norm h_{L^P(\tau)}
 =\norm{f\ind_{\{K>0\}}}_{L^P(\sigma)}
 \le\norm f_{L^P(\sigma)}.
\end{equation}
There is equality for supported inputs. The original integral ignores
$\{K=0\}$. The inverse substitution on that supported subspace
therefore preserves the operator norm. Truncation deals with infinite
kernel values before passage to the limit.
The strict upper-set version of Theorem 2.7 gives
\eqref{eq:generalQs}--\eqref{eq:generalQi}. In the supremal formula the
cut pairs outputs strictly below $x$ with inputs at or above $x$.
Equivalently one may use an inclusive output prefix and a strict input
tail. Taking one-sided limits over all real cuts gives the same
supremum. In the integral formula the output at $x$ is integrated
against $d\lambda(x)$, so its prefix is inclusive and its input tail
must be strict. No diagonal atom is counted. Restrictions to finite
measures, followed by monotone convergence, give the extended-valued
versions used here.

For $\calP_W$ substitute $(P,S,\sigma,\lambda)
=(\theta,t,\mu,W^t\nu)$. For $\calQ_W$ use
$(\theta,t,\mu,\nu)$ with $K=W$. Because
$\mu([a,x])^{1/\theta'}=\bar A_1(x)^q$, this gives
\eqref{eq:mixedPs}--\eqref{eq:mixedQi}. Positivity and $t>1$ imply
\[
 \max(\norm{\calP_W},\norm{\calQ_W})
 \le\norm{\calS_W}\le\norm{\calP_W}+\norm{\calQ_W},
\]
which proves the second equivalence in \eqref{eq:mixednorm}. This use of
Minkowski is in the auxiliary Banach target $L^t(\nu)$.

\subsection{Tails and the threshold}
The exponent division is
\begin{equation}\label{eq:mixedthreshold}
 \theta\le t\ \Longleftrightarrow\ \frac1q\le\frac p{p-q}
 \ \Longleftrightarrow\ p(1-q)\le q
 \ \Longleftrightarrow\ q\ge\frac p{p+1}.
\end{equation}
Suppose $q<p/(p+1)$ and $B$ is bounded. The full integrands in
\eqref{eq:mixedPi}--\eqref{eq:mixedQi} are then integrable. For a
lower-tail weight $w_c=u\ind_{(a,c)}$, its cumulative tail satisfies
$0\le W_c\le U$ and vanishes on $[c,b_u)$. At every ordinary fixed
$x>a$, both relevant inner tails are eventually zero. The outer
measure $\nu$ has no initial atom. Dominated convergence with the full
integrands therefore gives
$\mathfrak P_m(w_c),\mathfrak Q_m(w_c)\to0$.

For an upper-tail weight $w_d=u\ind_{(d,b_u)}$, its tail is
$W_d(x)=U(\max(x,d))$. Since $U(x_0)<\infty$ at one strict cut,
continuity from above gives $U(d)\downarrow0$, including when
$b_u=\infty$. Thus $W_d(x)\to0$ at each ordinary fixed $x$ and
$W_d\le U$. On points where their multiplying cumulative factors are
positive, the inner full integrals are finite almost everywhere by the
finite full characteristics. Dominated convergence applies first
inside and then outside. Points with a zero factor contribute zero
throughout. Both upper-tail characteristics consequently tend to zero.
The norm equivalence is uniform in these truncations, so the one-sided
operator tails vanish. Lemma~\ref{lem:sampler} and
Theorem~\ref{thm:inherited} prove \eqref{eq:mixedautomatic} and, without
the exponent restriction, \eqref{eq:mixedtails}. We do not invoke a
compactness-transfer theorem beyond its published exponent range.

For completeness consider $q\ge1$. For fixed $g$ the map in the first
input has kernel
$k_s(x)=v_1(s)^{-1}\ind_{(s,b_u)}(x)Hg(x)$ against $f(s)v_1(s)\,ds$.
Its $L^1\to L^q(w)$ norm is the essential supremum of the kernel-section
norms. The upper bound is Minkowski in the Banach target. The reverse
bound follows by approximating kernel sections by weighted normalized
inputs, or by the standard $L^1$ kernel norm theorem. Taking the
supremum over $g$ gives \eqref{eq:mixedBanachC}. The interchange with
the essential supremum is justified by a countable dense subset of
$L^p(v_2)$ and positivity, first on finite truncations. The same local
realization and inherited theorem give the tail criterion. This
completes the proof of Theorem~\ref{thm:mixed}, apart from sharpness,
which the following examples establish.

\begin{example}[Critical mixed powers]\label{ex:mixed}
On $(0,1)$ take $A,B_0>0$ and
\begin{equation}\label{eq:mixedpowers}
 v_1(x)=x^{-A},\qquad v_2(x)=x^{p-1-pB_0},\qquad
 u(x)=x^{-h-1},\qquad h=q(A+B_0).
\end{equation}
Here $A_1^-(x)=x^A$ and
$A_p(x)=(p'B_0)^{-1/p'}x^{B_0}$.
For $p/(p+1)\le q<1$, $W(x)\asymp x^{-h}$ near zero,
$d\mu=Arx^{Ar-1}\,dx$, and
$d\nu\asymp x^{B_0R-1}\,dx$.
The two supremal characteristics are finite. Their lower-end powers
are $x^{q(A+B_0)-h}=1$, while the upper endpoint is harmless.
For example, $ht-B_0R=qAt>0$ and
$h\theta'-Ar=qB_0\theta'>0$, so the inner integrals have precisely the
required lower-end powers. This verifies boundedness at the threshold
as well as above it in this range.

For the remaining range $q\ge1$ there is a direct boundedness check.
Choose a real $s>\max(p,q)$ and set
$1/t_0=1/q-1/s$. Then $t_0>1$,
$H:L^1(x^{-A})\to L^{t_0}(x^{-t_0A-1})$ and
$H:L^p(x^{p-1-pB_0})\to L^s(x^{-sB_0-1})$ are bounded by the scalar
supremal Hardy tests. H\"older with $q/t_0+q/s=1$ proves bilinear boundedness.

To see the nonvanishing tail directly, put
$f_\eps=c_A\eps^{A-1}\ind_{(\eps/2,\eps)}$ and
$g_\eps=c_{p,B_0}\eps^{B_0-1}\ind_{(\eps/2,\eps)}$, choosing the
positive constants to make the input norms one. Beyond $\eps$ their
primitives are positive constants times $\eps^A$ and
$\eps^{B_0}$. Hence
\[
 \int_\eps^{2\eps}|B(f_\eps,g_\eps)(x)|^qu(x)\,dx
 =c\eps^{q(A+B_0)-h}=c>0.
\]
Every lower tail contains such a window. By
Theorem~\ref{thm:inherited}, the bounded map is noncompact.
Below $p/(p+1)$ the same critical powers make the integral
characteristics diverge like $\int_0 dx/x$, and do not give a bounded
counterexample.

The concrete choice $p=2$, $q=3/4$, $A=B_0=1/2$ gives
$v_1=x^{-1/2}$, $v_2=1$, $u=x^{-7/4}$. It lies above the critical
value $2/3$ and is bounded but noncompact.
\end{example}

\section{Two \texorpdfstring{$L^1$}{L1} inputs}\label{sec:double}

Both inputs now have exponent $1$. Write
\[A_i^-(x)=\esssup_{a<s<x}v_i(s)^{-1},\qquad i=1,2,\]
and use their cumulative representatives $\bar A_i$. For $0<q<1$ put
\begin{equation}\label{eq:doubleparams}
 r=\frac q{1-q},\qquad\alpha=\frac1q,\qquad
 \beta=\frac1{1-q}=\alpha',\qquad d\mu_i=d(\bar A_i^r).
\end{equation}
Both measures may have initial or interior atoms. Given $0\le w\le u$,
let $W(x)=\int_x^{b_u}w$ and
\begin{equation}\label{eq:doubleS}
 \Phi_h(y)=\int_{[a,y]}h\dd\mu_1,\qquad
 \calS_Wh(x)=\int_x^{b_u}w(y)\Phi_h(y)\,dy.
\end{equation}
Tonelli gives the exact split
\begin{equation}\label{eq:doublesplit}
 \begin{aligned}\calP_Wh(x)&=W(x)\int_{[a,x]}h\dd\mu_1,\\
 \calQ_Wh(x)&=\int_{(x,b_u)}W(s)h(s)\dd\mu_1(s),
 \qquad\calS_W=\calP_W+\calQ_W.\end{aligned}
\end{equation}
In particular, a mass at $s=x$ occurs in $\calP_W$ and not in $\calQ_W$.

For $1/2\le q<1$ define
\begin{align}
 \mathfrak P_{\le}(w)
 &=\sup_{x\in[a,b_u)}\bar A_1(x)^q
 \left(\int_{[x,b_u)}W(y)^\beta\dd\mu_2(y)\right)^{1/\beta},
 \label{eq:doublePs}\\
 \mathfrak Q_{\le}(w)
 &=\sup_{x\in[a,b_u)}\mu_2([a,x))^{1/\beta}
 \left(\int_{[x,b_u)}W(y)^\beta\dd\mu_1(y)\right)^{1/\beta}.
 \label{eq:doubleQs}
\end{align}
The initial cut $x=a$ is included in \eqref{eq:doublePs}. For
$0<q<1/2$, put $\kappa=(1-2q)^{-1}$ and define
\begin{align}
 \mathfrak P_{>}(w)^\kappa
 &=\int_{[a,b_u)}
 \left(\int_{[x,b_u)}W(y)^\beta\dd\mu_2(y)\right)^{\kappa q}
 \mu_1([a,x])^{\kappa(1-q)}W(x)^\beta\dd\mu_2(x),
 \label{eq:doublePi}\\
 \mathfrak Q_{>}(w)^\kappa
 &=\int_{[a,b_u)}\mu_2([a,x])^{\kappa q}
 \left(\int_{(x,b_u)}W(y)^\beta\dd\mu_1(y)\right)^{\kappa(1-q)}
 \dd\mu_2(x).\label{eq:doubleQi}
\end{align}
Let
\begin{equation}\label{eq:doubleC}
 \mathcal C_d(w)=
 \begin{cases}
 (\mathfrak P_>(w)+\mathfrak Q_>(w))^{1/q},&0<q<1/2,\\
 (\mathfrak P_{\le}(w)+\mathfrak Q_{\le}(w))^{1/q},&1/2\le q<1,\\
 \displaystyle\sup_{a<x<b_u}W(x)^{1/q}A_1^-(x)A_2^-(x),&q\ge1.
 \end{cases}
\end{equation}
These are weight characteristics, with the same monotone-truncation
convention for zero factors as in Section~\ref{sec:mixed}.

\begin{theorem}[Main Theorem B]\label{thm:double}
Under the endpoint setting and local admissibility of
Section~\ref{sec:framework}, $B:L^1(v_1)\times L^1(v_2)\to L^q(u)$ is
bounded if and only if $\mathcal C_d(u)<\infty$. For $0<q<1$,
\begin{equation}\label{eq:doublenorm}
 \norm{B_w}^q\asymp_q
 \norm{\calS_W:L^\alpha(\mu_1)\to L^\beta(\mu_2)},\qquad
 \norm{B_w}\asymp_q\mathcal C_d(w),
\end{equation}
uniformly for $0\le w\le u$. For $q\ge1$,
$\norm{B_w}=\mathcal C_d(w)$ exactly. In the bounded case, compactness
is equivalent to
\begin{equation}\label{eq:doubletails}
 \mathcal C_d(u\ind_{(a,c)})\longrightarrow0\quad(c\downarrow a),
 \qquad
 \mathcal C_d(u\ind_{(d,b_u)})\longrightarrow0\quad(d\uparrow b_u).
\end{equation}
In particular,
\begin{equation}\label{eq:doubleautomatic}
 0<q<\frac12\quad\Longrightarrow\quad
 B\text{ bounded}\ \Longleftrightarrow\ B\text{ compact}.
\end{equation}
For every $q\ge1/2$ bounded noncompact examples exist. The fixed-window
finite-rank estimate for $0<q<1$ is
\[
 a_{N+1}^{(2)}(M_{c,d}B)\le C_q\norm B\,V^{1/q}N^{-1}.
\]
\end{theorem}

\subsection{Two endpoint reductions}
For fixed $g$, apply \eqref{eq:L1scalar}--\eqref{eq:tonelliidentity}
with the first weight. This gives
\[
 \sup_{\norm f_{L^1(v_1)}\le1}\norm{Hf\,Hg}_{L^q(w)}^q
 \asymp_q\norm{F_g}_{L^\beta(\mu_1)},\qquad
 F_g(x)=\int_x^{b_u}|Hg(y)|^qw(y)\,dy.
\]
The dual exponent of $\beta$ is $\alpha>1$. Hence positive Banach
duality in $L^\beta(\mu_1)$, followed by Tonelli, gives
\begin{equation}\label{eq:doubledual}
 \norm{F_g}_{L^\beta(\mu_1)}
 =\sup_{\substack{h\ge0\\\norm h_{L^\alpha(\mu_1)}\le1}}
 \int_{I_u}|Hg(y)|^qw(y)\Phi_h(y)\,dy.
\end{equation}
Now apply the scalar $L^1$ criterion to $g$ with input weight $v_2$
and output density $w\Phi_h$. Its cumulative output is precisely
$\calS_Wh$. Interchanging the independent suprema over $g$ and $h$
proves the first equivalence in \eqref{eq:doublenorm}. Both endpoint
steps have constants depending only on $q$. Each step uses a density
output before passing to its Stieltjes representation. No density
hypothesis is imposed on either $\mu_i$ afterward.

This derivation is valid for signed or complex inputs because the
operator is positive in modulus. Intermediate dual tests may be taken
nonnegative. Sigma-finiteness follows from the local profiles. Infinite
total masses are allowed. Finite restrictions and monotone convergence
justify the extended-valued identities. In particular, the reduction
does not pass through ordinary $L^1(I^2,v_1v_2)$. Its input supremum
remains the product of the two original unit balls.

To obtain the characteristics, apply
\eqref{eq:generalHs}--\eqref{eq:generalQi} with
\[
 (P,S,\sigma,\lambda)=(\alpha,\beta,\mu_1,W^\beta\mu_2)
 \quad\hbox{for }\calP_W,
\]
and with $(\alpha,\beta,\mu_1,\mu_2)$ and $K=W$ for $\calQ_W$.
The exponent division is
\begin{equation}\label{eq:doublethreshold}
 \alpha\le\beta\ \Longleftrightarrow\ \frac1q\le\frac1{1-q}
 \ \Longleftrightarrow\ q\ge\frac12.
\end{equation}
Also
$\mu_1([a,x])^{1/\alpha'}=\bar A_1(x)^{r/\beta}
=\bar A_1(x)^q$. This proves
\eqref{eq:doublePs}--\eqref{eq:doubleQs} in the supremal range.
In the integral range,
\[
 \frac1\kappa=\frac1\beta-\frac1\alpha=1-2q,
 \qquad\frac\kappa\alpha=\kappa q,
 \qquad\frac\kappa{\alpha'}=\kappa(1-q),
\]
which gives exactly \eqref{eq:doublePi}--\eqref{eq:doubleQi}.
Positivity bounds the norm of $\calS_W$ between the maximum and the
sum of the two component norms. This completes the boundedness
characterization for $q<1$.

\subsection{Why the atomic conventions affect the theorem}
If $\mu_1=m_1\delta_a$, then
$\calP_Wh(x)=m_1W(x)h(a)$ and $\calQ_W=0$. Its exact norm is
\begin{equation}\label{eq:pureinitial}
 \norm{\calP_W}=m_1^{1/\alpha'}
 \left(\int_{[a,b_u)}W(x)^\beta\dd\mu_2(x)\right)^{1/\beta}.
\end{equation}
When both $\mu_i=m_i\delta_a$, this becomes
$W(a)m_1^{1/\alpha'}m_2^{1/\beta}$. The cut $x=a$ in
\eqref{eq:doublePs} detects it, whereas a supremum restricted to $x>a$
would be zero in this pure-atom case.
If instead $\mu_2=m_2\delta_a$, the upper component has norm
\[
 m_2^{1/\beta}\left(\int_{(a,b_u)}W^\beta\dd\mu_1\right)^{1/\beta}.
\]
The strict exclusion of $a$ is essential. In the supremal form this
value is recovered by cuts approaching $a$ from the right. In the
integral form it is the contribution of the output atom itself.

The same issue occurs at an interior atom. Its lower contribution at
the cut is included, and its upper contribution at that cut is excluded.
For example, a finite ordered atomic model gives the lower-triangular
sum $\sum_{k\le j}h_k\mu_{1,k}$ and the strictly upper-triangular sum
$\sum_{k>j}W_kh_k\mu_{1,k}$. These are precisely the bracket patterns
in \eqref{eq:doublesplit}. No smoothing of an atom is used.
Singular continuous profiles are allowed as well. If $C$ is the Cantor
function and $\bar A(x)^r=c+x+\eps C(x)$ on $[0,1)$, $c,\eps>0$, the
weight $v=1/\bar A$ realizes this profile and the measure is
$c\delta_0+dx+\eps\,dC$. All three pieces enter the displayed
characteristics.

\subsection{Automatic compactness, including the initial output atom}
Assume $q<1/2$ and $B$ bounded. Thus both full integral
characteristics are finite. For the lower-tail weight $w_c=u\ind_{(a,c)}$
write $W_c(x)=\int_x^c u$ for $x<c$ and zero otherwise.
At an ordinary fixed $x>a$ the integrands in
\eqref{eq:doublePi}--\eqref{eq:doubleQi} eventually vanish. They are
bounded by the corresponding full integrands. Unlike the mixed case,
$\mu_2$ may charge $a$, so this point requires a separate check.

Write $m_i=\mu_i(\{a\})$. If $m_2=0$ there is no extra output term.
If $m_2>0$ and $m_1>0$, \eqref{eq:doublePi} forces $U(a)<\infty$.
The initial lower contribution already contains both masses and
$U(a)$. Consequently $W_c(a)=\int_a^cu\to0$. The integral of
$U^\beta$ against $\mu_2$ is finite by the same nonzero initial
factor, so dominated convergence also applies to its inner tail.
If $m_2>0$ and $m_1=0$, the lower contribution at $a$ is identically
zero. For the upper contribution, finiteness of
\eqref{eq:doubleQi} at this atom forces
\[
 \int_{(a,b_u)}U(y)^\beta\dd\mu_1(y)<\infty.
\]
As $W_c(y)\to0$ for every $y>a$, dominated convergence makes this
strict inner tail vanish. This also handles the upper contribution
when both initial masses are positive. Thus no residual atom remains.
The argument is unchanged when $a=-\infty$, interpreted in the
adjoined ordered space.

For upper tails $W_d(x)=U(\max(x,d))\to0$, including at the adjoined
point since $W_d(a)=U(d)\downarrow0$. Full integrability again
provides domination. At zero cumulative factors the contribution is
zero for every truncation, and at positive factors the full inner
integral is finite almost everywhere. We may therefore apply dominated
convergence to inner and outer integrals, with the initial cases
handled exactly as above. Both one-sided characteristics tend to zero.
By Lemma~\ref{lem:sampler} and Theorem~\ref{thm:inherited}, this proves
\eqref{eq:doubleautomatic}. The uniform norm characterization proves
\eqref{eq:doubletails} throughout $q<1$, and the same lemma gives the
local approximation estimate in the theorem.

For $q\ge1$, weighted $L^1$ kernel integration in both inputs gives
\begin{equation}\label{eq:doublekernel}
 \norm{B_w}=\esssup_{s,t\in I}
 \frac{W(\max(s,t))^{1/q}}{v_1(s)v_2(t)}.
\end{equation}
One can justify this equality first with bounded kernels on finite
measure sets and then by monotone truncation. The upper bound uses
Minkowski in $L^q(w)$, and the reverse bound uses normalized indicators
on positive-measure sets approaching kernel-section norms.
For $s,t<x$, monotonicity gives
$W(x)\le W(\max(s,t))$, yielding the lower bound by the last line
of \eqref{eq:doubleC}. Conversely take $x\downarrow\max(s,t)$ and
use continuity of $W$ at ordinary strict cuts, together with the
almost-everywhere profile inequalities. This proves equality with
$\mathcal C_d(w)$. The same local realization and inherited theorem
supply the compactness criterion in this range. The next examples
complete the sharpness assertions.

\begin{example}[Bounded point testing with a divergent Stieltjes norm]
\label{ex:divergence}
On $(0,1)$ let
\[
 v_1=1,\qquad v_2=x^{-1},\qquad u=x^{-q-1},\qquad0<q<1.
\]
Here $A_1^-=1$, $A_2^-=x$,
$\mu_1=\delta_0$, and $d\mu_2=rx^{r-1}\,dx$. Moreover
\[
 W(x)=\frac{x^{-q}-1}{q},\qquad
 W(x)^{1/q}A_1^-(x)A_2^-(x)
 =q^{-1/q}(1-x^q)^{1/q}\le q^{-1/q}.
\]
Thus the elementary point test is bounded. However, $h(0)=1$ has
$L^\alpha(\mu_1)$ norm one and $\calS_Wh=W$. Near zero,
\begin{equation}\label{eq:divergentdensity}
 W(x)^\beta\dd\mu_2(x)
 \asymp_q x^{-q\beta+r-1}\,dx=\frac{dx}{x},
 \qquad q\beta=r.
\end{equation}
The logarithmic divergence proves $\calS_Wh\notin L^\beta(\mu_2)$,
so $B$ is unbounded. For $q\ge1/2$ this is detected by
\eqref{eq:doublePs} at the initial cut. For $q<1/2$ it follows directly
from the operator reduction, hence also from the integral criterion.
Omitting the initial mass would replace $\mu_1$ by zero and lose the
entire obstruction. This example concerns failure of boundedness,
not bounded noncompactness.
\end{example}

\begin{example}[Noncompactness at the critical value]\label{ex:critical}
Take $q=1/2$, $v_1=v_2=x^{-1}$, $u=x^{-2}$ on $(0,1)$.
For nonnegative $f$ Tonelli gives
\[
 \int_0^1Hf(x)x^{-2}\,dx
 =\int_0^1f(t)(t^{-1}-1)\,dt\le\norm f_{L^1(x^{-1})}.
\]
Cauchy--Schwarz in the $q$th-power integral therefore gives
\[
 \int_0^1\sqrt{Hf(x)Hg(x)}x^{-2}\,dx
 \le\norm f_{L^1(x^{-1})}^{1/2}\norm g_{L^1(x^{-1})}^{1/2}.
\]
Squaring proves boundedness into $L^{1/2}(x^{-2})$.
Let
\[
 f_\eps=g_\eps=(\log2)^{-1}\ind_{(\eps/2,\eps)}.
\]
Both input norms are one, since
$(\log2)^{-1}\int_{\eps/2}^{\eps}dt/t=1$. For $x\ge\eps$,
$Hf_\eps(x)=Hg_\eps(x)=\eps/(2\log2)$. Therefore
\begin{equation}\label{eq:criticaldefect}
 \int_\eps^{2\eps}|B(f_\eps,g_\eps)(x)|^{1/2}x^{-2}\,dx
 =\frac1{4\log2}.
\end{equation}
Every sufficiently small lower tail contains one such interval, so
its operator norm is at least $(4\log2)^{-2}$. The positive tail
defect and Theorem~\ref{thm:inherited} give noncompactness.

More generally, with $v_1=v_2=x^{-1}$ and $u=x^{-2q-1}$,
$B$ is bounded and noncompact for every $q\ge1/2$.
Indeed the scalar $L^1(x^{-1})\to L^{2q}(x^{-2q-1})$ Hardy norm is
finite by its Banach-target point test. H\"older yields the bilinear
bound. The same inputs make the integral over $(\eps,2\eps)$ a
positive constant independent of $\eps$. Thus nonautomatic compactness
holds throughout the remaining range, without asserting that all
bounded operators there are noncompact.
\end{example}

\section{\texorpdfstring{$L^\infty$}{Linfty} inputs and scalar transfer}\label{sec:infinity}

Throughout this section $a>-\infty$ and $0<q<\infty$. The
$L^\infty$ norm is the ordinary essential supremum. This left-endpoint
restriction is a well-definedness hypothesis. On $a=-\infty$ the unit
input $f=1$ has $Hf(x)=\infty$ at every ordinary cut. A full
$L^\infty$ endpoint on that interval is consequently outside the
present theorem for a nontrivial target.

For $1\le p<\infty$ define
\begin{equation}\label{eq:scalarcompanion}
 T_pg(x)=(x-a)Hg(x),\qquad \phi_f(x)=\frac{Hf(x)}{x-a},
 \qquad\widetilde u(x)=(x-a)^qu(x).
\end{equation}
Thus $B(f,g)=\phi_fT_pg$. Multiplication by $x-a$ identifies
$H:L^p(v)\to L^q(\widetilde u)$ isometrically with
$T_p:L^p(v)\to L^q(u)$.

\begin{theorem}[Main Theorem C]\label{thm:infinity}
Assume $a>-\infty$, $1\le p<\infty$, and the local evaluation
hypotheses of Section~\ref{sec:framework}. The map
$B:L^\infty\times L^p(v)\to L^q(u)$ is bounded if and only if $T_p$
is bounded, and for every measurable output set $E$,
\begin{equation}\label{eq:setidentity}
 \norm{\ind_EB}=\norm{\ind_ET_p}.
\end{equation}
The two maps are compact simultaneously. In the bounded case,
\begin{equation}\label{eq:transfer}
 \norm B_\e=\norm{T_p}_\e,\qquad
 \distfr(B)=\distfr(T_p),\qquad\chi_q(B)=\chi_q(T_p),
\end{equation}
where scalar quantities use bounded linear approximants. For $q<1$
the complete-metric quantities agree as well. Their approximation
numbers satisfy
\begin{equation}\label{eq:approxtransfer}
 a_n(T_p)\le a_n^{(2)}(B),\qquad
 \lim_n a_n(T_p)=\lim_n a_n^{(2)}(B).
\end{equation}
For $1<p<\infty$, boundedness implies compactness when $q<p$.
For $q\ge p$ compactness is nonautomatic. This theorem explicitly
includes $p=1$. Both $(\infty,1)$ and $(1,\infty)$ reduce to
$H:L^1(v)\to L^q((x-a)^qu)$, with automatic compactness for $q<1$
and nonautomatic compactness for $q\ge1$. Input interchange gives all
symmetric assertions.
\end{theorem}

The elementary estimate $|\phi_f|\le1$ on the $L^\infty$ unit ball
gives one inequality in \eqref{eq:setidentity}. The input $f=1$ gives the other.
In particular, the scalar and bilinear maps have exactly the same
active tails. To obtain the quantitative conclusions from this fact,
we still need local finite-rank realization of both maps.

\begin{lemma}[Linear nodal realization]\label{lem:nodal}
If $T_p$ is bounded, each $M_{c,d}T_p$ is an operator-norm limit of
finite-rank linear maps. Each $M_{c,d}B$ is an operator-norm limit of
finite-rank bilinear maps. More precisely, a scalar sampler of rank at
most $N$ and a nodal multiplier interpolant with $M+1$ nodes give a
bilinear map of rank at most $(M+1)N$.
\end{lemma}

\begin{proof}
On $[c,d]\Subset I_u$, $Hf$ is absolutely continuous and, for almost
every $x$,
\[
 \phi_f'(x)=\frac{(x-a)f(x)-Hf(x)}{(x-a)^2},\qquad
 |\phi_f'(x)|\le\frac2{c-a}\quad(\norm f_\infty\le1).
\]
Let $c=x_0<\cdots<x_M=d$ and let $\ell_0,\ldots,\ell_M$ be the
usual nonnegative piecewise-linear hat functions. The nodal operator
\begin{equation}\label{eq:interpolant}
 I_{\mathcal P}\phi_f=\sum_{j=0}^M\phi_f(x_j)\ell_j
\end{equation}
is linear in $f$, contractive in the supremum norm, and satisfies
\begin{equation}\label{eq:interpolationerror}
 \sup_{\norm f_\infty\le1}\norm{I_{\mathcal P}\phi_f-\phi_f}_\infty
 \le\frac{2|\mathcal P|}{c-a}=:\eps_{\mathcal P}\longrightarrow0.
\end{equation}
This follows by writing the interpolant on each cell as a convex
combination of its endpoint values and using the Lipschitz bound.
The choice of approximant is a fixed linear operation on each input.

For the scalar map, the evaluation test gives
\begin{equation}\label{eq:scalarpoint}
 U(x)^{1/q}(x-a)A_p(x)\le\norm{T_p},
\end{equation}
with $A_1^-$ when $p=1$. Indeed, for a nonnegative $g$ supported before
$x$ and $y\ge x$, $(y-a)Hg(y)\ge(x-a)Hg(x)$.
Choose a partition $c=y_0<\cdots<y_N=d$, $J_k=(y_{k-1},y_k)$, with
maximal $\nu_0$-mass $\delta_N\to0$, where $\nu_0$ is
\eqref{eq:gauge}. Put
\begin{equation}\label{eq:scalarsampler}
 R_Ng(x)=\sum_{k=1}^N(x-a)Hg(y_{k-1})\ind_{J_k}(x).
\end{equation}
This is a bounded linear operator of rank at most $N$. For a unit input
let $m_k=\int_{J_k}|g|^pv$. Using the scalar evaluation inequality on
the increment and \eqref{eq:scalarpoint},
\[
 |T_pg(x)-R_Ng(x)|
 \le\norm{T_p}U(x)^{-1/q}m_k^{1/p}\qquad(x\in J_k).
\]
Therefore
\begin{equation}\label{eq:scalarerror}
 \norm{M_{c,d}T_pg-R_Ng}_q^q
 \le\norm{T_p}^q\sum_k\nu_0(J_k)m_k^{q/p},\qquad\sum_km_k\le1.
\end{equation}
The allocation calculation \eqref{eq:allocation} bounds the last sum
by $\delta_N^{q/p}V^{1-q/p}$ when $q<p$, and by $\delta_N$ when
$q\ge p$. Both tend to zero. This proves the scalar assertion, including
$p=1$ and $q<1$.

Extend the window functions by zero and define
\begin{equation}\label{eq:productapproximant}
 F_{\mathcal P,N}(f,g)=(I_{\mathcal P}\phi_f)(R_Ng).
\end{equation}
Each coefficient is linear in its corresponding input, so this map is
bilinear. Its output span is contained in
\[
 \operatorname{span}\{\ell_j(x)(x-a)\ind_{J_k}(x):
                 0\le j\le M,\ 1\le k\le N\}.
\]
All these functions belong to $L^q(u)$ on the fixed window, and the
rank is at most $(M+1)N$ (linear dependencies can only lower it).
The decomposition
\[
 M_{c,d}B-F_{\mathcal P,N}
 =(\phi_f-I_{\mathcal P}\phi_f)M_{c,d}T_pg
 +(I_{\mathcal P}\phi_f)(M_{c,d}T_pg-R_Ng)
\]
and $\vartheta=\min(1,q)$ give
\begin{equation}\label{eq:producterror}
 \norm{M_{c,d}B-F_{\mathcal P,N}}^\vartheta
 \le\eps_{\mathcal P}^\vartheta\norm{T_p}^\vartheta
       +\norm{M_{c,d}T_p-R_N}^\vartheta\longrightarrow0.
\end{equation}
Only multiplication by uniformly bounded functions and the valid
$\vartheta$-triangle inequality are used. This proves the lemma.
\end{proof}

\begin{proof}[Completion of Theorem~\ref{thm:infinity}]
Lemma~\ref{lem:nodal} verifies the local hypotheses. Apply
Theorem~\ref{thm:inherited} to $B$ and to the bilinear representation
$(g,z)\mapsto zT_pg$, $z\in\mathbb R$ or $\mathbb C$. This scalar
representation has the same unit-ball image as $T_p$. Restrictions at
$z=1$ and extensions by multiplication with $z$ identify compact and
finite-rank linear approximants with their bilinear counterparts and
preserve their distances. The common tails from
\eqref{eq:setidentity} therefore give compactness equivalence and
\eqref{eq:transfer}.
For any rank-$<n$ bilinear approximant $F$ to $B$, the map
$g\mapsto F(1,g)$ is linear of rank $<n$, and
\[
 \norm{T_p-F(1,\cdot)}\le\norm{B-F}.
\]
This proves the inequality in \eqref{eq:approxtransfer}. Its limit
statement follows from the inherited finite-rank limit for the two
locally realized maps. No termwise equality is inferred.

It remains to identify the automatic regimes using scalar background.
For $1<p<\infty$ and $q<p$, set $s=pq/(p-q)$ and
$\widetilde U(x)=\int_x^{b_u}\widetilde u$. The scalar Hardy norm is
equivalent to
\begin{equation}\label{eq:scalarintegral}
 \left(\int_{I_u}\widetilde U(x)^{s/p}
                 A_p(x)^s\widetilde u(x)\,dx\right)^{1/s}.
\end{equation}
For $p=1$, $q<1$, its criterion is
\begin{equation}\label{eq:scalaroneintegral}
 \left(\int_{I_u}(A_1^-(x)\widetilde U(x))^{q/(1-q)}
                   \widetilde u(x)\,dx\right)^{(1-q)/q}.
\end{equation}
These are the classical density criteria
\cite{SinnamonStepanov1996,Santacruz2024}. For a shrinking output tail,
replace $\widetilde u$ by $\widetilde u\ind_E$ and its cumulative by
$\widetilde U_E\le\widetilde U$. The full integrable expression
bounds the truncated integrand. The outer density factor tends to zero
almost everywhere. Dominated convergence proves vanishing scalar tails.
The local scalar and bilinear realizations then prove automatic
compactness as stated. For $q\ge p>1$ the scalar boundedness criterion
is the usual supremum of
$\widetilde U(x)^{1/q}A_p(x)$. For $p=1$, $q\ge1$ it is the exact
kernel-section supremum. Compactness requires their truncated tail
versions to vanish. The example below shows that it is nonautomatic.
\end{proof}

\begin{corollary}[Two ordinary $L^\infty$ inputs]\label{cor:doubleinfty}
If $a>-\infty$, then
\begin{equation}\label{eq:doubleinfnorm}
 \norm{B:L^\infty\times L^\infty\to L^q(u)}
 =\left(\int_I(x-a)^{2q}u(x)\,dx\right)^{1/q}.
\end{equation}
For every finite $q>0$, this map is bounded if and only if it is compact.
In the bounded case all the distances and noncompactness quantities
in \eqref{eq:transfer} are zero.
\end{corollary}

\begin{proof}
The envelope $|HfHg|\le(x-a)^2$ on the unit balls proves the upper
norm bound, and $f=g=1$ gives equality. On a strict window both
primitive families are bounded and uniformly Lipschitz. Applying fixed
linear nodal interpolation to each primitive and multiplying the
interpolants produces finite-rank bilinear maps, with output span
contained in the products of their hat functions. Uniform convergence
on the window implies convergence in $L^q(u)$ because $U(c)<\infty$.
If \eqref{eq:doubleinfnorm} is finite, its integrable envelope makes the
two operator tails vanish. Theorem~\ref{thm:inherited} completes the
proof.
\end{proof}

\begin{example}[Nonautomatic scalar faces]\label{ex:scalarfaces}
On $(0,1)$ let $p\ge1$, $q\ge p$, and take
\[
 v(x)=(1-x)^{2p-1},\qquad
 \widetilde u(x)=(1-x)^{q-1},\qquad
 u(x)=x^{-q}(1-x)^{q-1}.
\]
For $p>1$, $A_p(x)\asymp(1-x)^{-1}$ near $1$. For $p=1$,
$A_1^-(x)=(1-x)^{-1}$. Also
$\widetilde U(x)^{1/q}=q^{-1/q}(1-x)$.
The scalar supremal Hardy test is finite on $(0,1)$ and stays bounded
away from zero near $1$. For every upper output cutoff, cuts lying
above it retain this lower bound on the operator norm. The scalar
operator and hence the bilinear $L^\infty$ face are bounded but
noncompact. This includes both mixed $L^1/L^\infty$ corners when $p=1$.
\end{example}

\section{Quantitative boundary synthesis}\label{sec:synthesis}

All boundary realizations have now been proved. Lemma~\ref{lem:sampler}
handles the finite-input faces, Lemma~\ref{lem:nodal} the faces with one
$L^\infty$ input, and Corollary~\ref{cor:doubleinfty} the remaining
corner. Their quantitative conclusion is a direct application of the
localization infrastructure, not an additional abstract theorem.

\begin{corollary}[Exact quantities on the input boundary]\label{cor:synthesis}
Let $1\le p_1,p_2\le\infty$, with at least one input exponent equal to
$1$ or $\infty$, and let $0<q<\infty$. Assume the endpoint setting of
Section~\ref{sec:framework}, including $a>-\infty$ on every
$L^\infty$ face. If $B:X_1\times X_2\to L^q(u)$ is bounded, then
\begin{equation}\label{eq:masteridentity}
 \norm B_\e=\distfr(B)=\chi_q(B)=\tactive(B).
\end{equation}
Moreover,
\begin{equation}\label{eq:masterlimit}
 a_n^{(2)}(B)\downarrow\tactive(B),\qquad
 B\text{ compact}\ \Longleftrightarrow\ \tactive(B)=0.
\end{equation}
For $0<q<1$ the complete-metric radius is
\begin{equation}\label{eq:mastermetric}
 \chi_\rho(B)=\tactive(B)^q.
\end{equation}
\end{corollary}

\begin{proof}
The local finite-rank constructions just cited verify the hypotheses
of Theorem~\ref{thm:inherited}, that is,
\cite[Theorems 5.1 and 5.3]{B1}. Applying it on $I_u$ gives all the
assertions with the radius conventions of Section~\ref{sec:framework}.
\end{proof}

To connect the exact identity to the weight formulas, let
\[
 d_a=\lim_{c\downarrow a}\norm{\ind_{(a,c)}B},\qquad
 d_b=\lim_{d\uparrow b_u}\norm{\ind_{(d,b_u)}B}.
\]
The inherited tail comparison reads
\begin{equation}\label{eq:tailsandwich}
 \max(d_a,d_b)\le\tactive(B)\le(d_a^q+d_b^q)^{1/q}.
\end{equation}
On the mixed and double-$L^1$ faces, the characteristics in
Theorems~\ref{thm:mixed} and \ref{thm:double}, evaluated at the
truncated output weights, estimate these defects with constants
independent of the cuts. They also estimate the coupled tail directly
by taking $w=u\ind_{(a,c)\cup(d,b_u)}$. Exactness in
\eqref{eq:masteridentity} refers to the operator-norm coupled tail.
It does not turn equivalent weight characteristics into exact norms.
On the $L^\infty$ faces, \eqref{eq:setidentity} gives an exact scalar
tail instead.

\subsection{The eight boundary cells}
Table~\ref{tab:phase} collects the automatic-compactness regimes.
Here ``nonautomatic'' means that bounded noncompact examples exist,
while individual bounded operators are compact precisely when their
active tails vanish. In the rows containing $p$, $1<p<\infty$.

\begin{table}[htbp]
\centering
\caption{Finite-$q$ input boundary. All $L^\infty$ rows require $a>-\infty$.}
\label{tab:phase}
\small
\setlength{\tabcolsep}{4pt}
\begin{tabular}{@{}cllll@{}}
\toprule
Cell & Inputs & Automatic & Nonautomatic & Mechanism\\
\midrule
E1 & $(1,p)$ & $q<p/(p+1)$ & $q\ge p/(p+1)$ & Mixed Stieltjes\\
E2 & $(p,1)$ & $q<p/(p+1)$ & $q\ge p/(p+1)$ & Symmetry\\
E3 & $(1,1)$ & $q<1/2$ & $q\ge1/2$ & Double Stieltjes\\
E4 & $(\infty,p)$ & $q<p$ & $q\ge p$ & Scalar transfer\\
E5 & $(p,\infty)$ & $q<p$ & $q\ge p$ & Symmetry\\
E6 & $(\infty,\infty)$ & All finite $q$ & None & Explicit norm\\
E7 & $(1,\infty)$ & $q<1$ & $q\ge1$ & Scalar $L^1$\\
E8 & $(\infty,1)$ & $q<1$ & $q\ge1$ & Symmetry\\
\bottomrule
\end{tabular}
\end{table}

Except at the double-$L^\infty$ corner, these boundaries can be written
as the familiar harmonic relation
\[
 q_c=\left(\frac1{p_1}+\frac1{p_2}\right)^{-1},\qquad\frac1\infty=0.
\]
The relation itself is not a new exponent law. The conclusions here
identify when boundedness forces compactness on the input boundary,
and when the endpoint-tail obstruction can persist. At
$(\infty,\infty)$, the reciprocal sum is zero and
Corollary~\ref{cor:doubleinfty} gives automatic compactness for every
finite target exponent.

The use of $b_u$ can matter even on a finite geometric interval.
Extend the weights of Example~\ref{ex:scalarfaces} from $(0,1)$ to
$(0,2)$ by arbitrary positive finite input weights and set $u=0$ on
$(1,2)$. The active endpoint is still $1$, and the noncompactness
concentration still occurs immediately below $1$. Conditions taken
only near the geometric endpoint $2$ would miss it. Conversely no
condition on the input primitives beyond $1$ is required by this
output norm.

\subsection{Relation to the interior theory and scope}
Kanjilal \cite{B1} supplies the interior theory for
$1<p_1,p_2<\infty$. Together with the boundary results above, this yields
finite-$q$ localization and quantitative identities over the closed
input square under the applicable hypotheses of the two papers. The
present paper supplies the input-boundary results. The proofs establishing
the boundary hypotheses are the Stieltjes reductions and local
realizations developed here.

The analysis is restricted to finite $q$. The case $q=\infty$ is excluded.
A full $L^\infty$ input with $a=-\infty$ is not covered without
additional structure. More general kernels and higher multilinearity
are outside the scope of this article.

\section*{Declarations}
\paragraph{Data availability.}
No empirical dataset is used in this mathematical study.
\paragraph{Funding.}
The author received no funding for this work.
\paragraph{Competing interests.}
The author declares no competing interests.
\paragraph{Generative AI disclosure.}
OpenAI's ChatGPT, including ChatGPT Work, and Codex were used as assistive tools during manuscript preparation for mathematical consistency checking, language refinement, readability improvement, and LaTeX preparation. The author independently reviewed the mathematical arguments, references, and final manuscript and takes full responsibility for all statements, proofs, and conclusions.

\bibliographystyle{amsplain}
\bibliography{references}
\end{document}